 \documentclass[a4paper,11pt]{article}
\usepackage{pdfsync}

 \usepackage[plainpages=false]{hyperref}
 \usepackage{amsfonts,amsmath,amssymb,amsthm}
 \usepackage{latexsym,lscape,rawfonts,mathrsfs}

 \usepackage[dvips]{color}
 \usepackage{multicol}

 \usepackage[all]{xy}
 \usepackage{eufrak}
 \usepackage{makeidx}
 \usepackage{graphicx,psfrag}
 \usepackage{pstool}

 \usepackage{array,tabularx}

 \usepackage{setspace}

 \usepackage[titletoc,title]{appendix}

\usepackage{txfonts}

 \newcommand{\R}{\ensuremath{\mathbb{R}}}

 \newcommand{\RP}{\ensuremath{\mathbb{RP}}}

 \newcommand{\ba}{\begin{align*}}
 \newcommand{\ea}{\end{align*}}

 \newcommand{\norm}[2]{{ \ensuremath{\left\|} #1 \ensuremath{\right\|}}_{#2}}
 
 \newcommand{\sconv}{\ensuremath{ \stackrel{C^\infty}{\longrightarrow}}}

 \makeatletter
 \def\ExtendSymbol#1#2#3#4#5{\ext@arrow 0099{\arrowfill@#1#2#3}{#4}{#5}}
 
 \makeatother

 \makeatletter
 \def\ExtendSymbol#1#2#3#4#5{\ext@arrow 0099{\arrowfill@#1#2#3}{#4}{#5}}
 \newcommand\longright[2][]{\ExtendSymbol{-}{-}{\rightarrow}{#1}{#2}}
 \makeatother

\numberwithin{equation}{section}

\newtheorem{thm}{Theorem}[section]

\newtheorem{prop}[thm]{Proposition}
\newtheorem{lem}[thm]{Lemma}
\newtheorem{conj}[thm]{Conjecture}
\newtheorem{quest}[thm]{Question}

\newtheorem{defn}[thm]{Definition}

\title{The rigidity of Ricci shrinkers of dimension four}
\author{Yu Li \quad and \quad Bing Wang\footnote{Both authors are partially supported by NSF grant DMS-1510401. They also acknowledge the invitation to MSRI Berkeley in spring 2016 supported by NSF grant DMS-1440140, where part of this work
has been carried out.}}
\date{\today}

\begin{document}
\maketitle

\begin{abstract}
In dimension $4$,  we show that a nontrivial flat cone cannot be approximated by smooth Ricci shrinkers with bounded scalar curvature and Harnack inequality, under the pointed-Gromov-Hausdorff topology.  
As applications, we obtain uniform positive lower bounds of scalar curvature and potential functions on Ricci shrinkers satisfying some natural geometric properties.
\end{abstract}

\tableofcontents

\section{Introduction}
A Ricci shrinker $(M, g, f)$ is a complete Riemannian manifold $(M,g)$ together with a smooth function $f: M \to \mathbb R$ such that
\begin{equation} 
\text{Rc}+\text{Hess}_f=\frac{1}{2}g.
\label{E100}
\end{equation}
Ricci shrinker is also called as gradient shrinking Ricci soliton. 
Direct calculation shows that $\nabla ( R+|\nabla f|^2-f)=0$.
Adding $f$ by a constant if necessary, we assume throughout that 
\begin{align} 
R+|\nabla f|^2=f.
\label{E101}
\end{align}
Under this normalization condition, it is known(c.f. Theorem~\ref{thm:MI03_1}) that
\begin{equation} \label{E105}
\int_M e^{-f}(4 \pi)^{-n/2} \,dv=e^{\mu},
\end{equation}
where $\mu=\mu(g,1)$ is the entropy functional of Perelman(c.f.~\cite{Pe1}).

The Ricci shrinker (\ref{E100}) was introduced by Hamilton~\cite{Ha88} in mid 1980's. 
As critical points of Perelman's $\mu$-entropy, the Ricci shrinkers play important roles in the singularity analysis of the Ricci flow.
For example, it is proved by Enders-M\"{u}ller-Topping~\cite{EMT} that the proper rescaling limit of a type-I singularity is always a nontrivial Ricci shrinker. 
More information and references can be found in Chapter 30 of the book~\cite{CCGG} by Chow and his coauthors. 
In dimension 2,  the only Ricci shrinkers are $\R^2$, $S^2$ and $\RP^2$ with standard metrics, due to the classification of Hamilton~\cite{Ha95}.
In dimension 3,  based on the breakthrough of Perelman(\cite{Pe1},~\cite{Pe2}), through the efforts of Naber~\cite{Naber},  Ni-Wallach~\cite{NW} and  Cao-Chen-Zhu~\cite{CCZ}, etc, 
we know that $\R^3$, $S^2 \times \R$, $S^3$ and their quotients are all the possible Ricci shrinkers. 

In dimension 4 and higher, much fewer is known about Ricci shrinkers.  Typically, some extra conditions of the curvature operator(e.g.~\cite{BoWi},~\cite{CWZ},~\cite{CaoChen},~\cite{CWY}, etc), 
or geometric properties at infinity(c.f.~\cite{KW15},~\cite{MW16}) are required to draw definite geometry conclusion.  We refer the readers the surveys~\cite{CaoH10},~\cite{CaoH11} for more detailed picture. 
Without such extra conditions, it is still not clear how to classify Ricci shrinkers. However, one can study the moduli of Ricci shrinkers.
In~\cite{CaoSe07}, Cao-Sesum showed the weak compactness of the moduli of K\"ahler Ricci solitons with uniformly bounded diameter and uniformly lower bound of Ricci curvature and $\mu$-functional.
The extra conditions were gradually weakened or removed by X. Zhang~\cite{ZhangXi}, Weber~\cite{Weber}, Chen-Wang~\cite{CW3}, Z.L. Zhang~\cite{ZhangZL} and Haslhofer-M\"{u}ller~\cite{HM11}, etc.

In this article, we focus on the study of the moduli $\mathcal{M}^*(A,H)$, which consists of 4d Ricci shrinkers which have uniform entropy lower bound and Harnack inequality of scalar curvature on unit ball(c.f. Definition~\ref{dfn:MH23_1} for precise definition).
Moreover, we require each Ricci shrinker has bounded(not uniformly) scalar curvature. 
In light of the results of Haslhofer-M\"{u}ller(c.f.~Theorem~\ref{T101}), it is known that such moduli has weak compactness.
In other words, any sequence  of Ricci shrinkers in $\mathcal{M}^*(A,H)$ sub-converges to an orbifold Ricci shrinker with locally finite singularities, in the pointed-Gromov-Hausdorff topology.   
Now we reverse the process and ask the following question:

\textit{What kind of orbifolds can be approximated by a sequence of Ricci shrinkers in $\mathcal{M}^*(A,H)$?}

\noindent
Note that flat cones $\R^{4}/\Gamma$ are naturally the simplest orbifold in dimension 4.
Therefore, the first step toward the solution of the above question is to check whether the flat cones can be approximated by Ricci shrinkers.
We completely solve this step by the following theorem, which is the main result of this article. 

\begin{thm}[\textbf{Gap Theorem}]
For any $A>0$ and $H>0$, there exists a small positive number $\epsilon=\epsilon(A,H)$ with the following property.

Suppose $(M, p, g, f) \in \mathcal{M}^*(A,H)$. Then we have
\begin{align}
  d_{PGH} \left\{ (M,p,g), (\R^4/\Gamma,0,g_{E})\right\}>\epsilon
\end{align} 
for every finite subgroup $\Gamma \subset O(4)$ acting freely on $S^3$. 
\label{thm:1}
\end{thm}

Note that Theorem~\ref{thm:1} can be illustrated as an $\epsilon$-regularity theorem. Namely, suppose
 \begin{align*}
 d_{PGH} \left\{ (M,p,g), (\R^4/\Gamma,0,g_{E})\right\}<\epsilon,
 \end{align*} 
 then we have uniform curvature, injectivity radius estimate inside the unit ball.  
 Such type statement was used in the literature of studying Einstein manifolds, e.g., in Cheeger-Colding-Tian~\cite{CCT} and Chen-Donaldson~\cite{CD13}, etc. 
 However, an essential difference here is that we do not allow rescaling of the metric since rescaling will destroy the structure of (\ref{E100}). 
 Theorem~\ref{thm:1} was motivated by the work of Biquard~\cite{Biquard}, Morteza-Viaclovsky~\cite{MoVi}, where they study whether an orbifold can be approximated by
 Einstein metrics, with some extra assumption of the topology of the underlying manifolds.

As applications of Theorem~\ref{thm:1}, we can uniformly estimate the scalar curvature $R$ and the Ricci potential function $f$. 

\begin{thm}[\textbf{Uniform positive lower bounds of scalar curvature and potential function}]
For any $A>0$ and $H>0$, there exist positive constants $C_a, C_b, C_c$ depending on $A$ and $H$ only with the following properties.

Suppose $(M, p, g, f) \in \mathcal{M}^*(A,H)$.  Then the following uniform estimates hold.

\begin{itemize}
\item[(a).] At base point $p$, we have
\begin{equation}\label{E107}
f \ge C_a.
\end{equation}
\item[(b).] In the ball $B(p,1)$, we have
\begin{equation}\label{E108}
R \ge C_b.
\end{equation}
\item[(c).] On the whole manifold $M$, we have
\begin{equation}\label{E109}
Rf \ge C_c.
\end{equation}
\end{itemize}

\label{thm:2}
\end{thm}

Theorem~\ref{thm:2} seems to be the first uniform  positive lower bound estimates of $R$ and $f$ for Ricci shrinkers.
Note that their upper bound and nonnegative lower bound are well known in literature(c.f. equation (\ref{E101}), inequality (\ref{E106}) and the reference nearby). \\

We remark that Theorem~\ref{thm:1} should be useful in the study of 4d Ricci flow singularities with bounded positive scalar curvature.  
For every such singularity, it seems natural that all the possible singularity model locate in the closure of the moduli $\mathcal{M}(A,H)$, as both requirements in the definition of $\mathcal{M}(A,H)$ are satisfied automatically.  \\

We briefly discuss the proof of Theorem~\ref{thm:1} and Theorem~\ref{thm:2}.  
The foundation of Theorem~\ref{thm:1} is a rigidity theorem(c.f. Theorem~\ref{AT02}) for singular Eigenfunctions of the drifted Laplace operator.
Suppose $v$ is a positive function satisfying
\begin{align}
    \Delta_f v(x)= v(x), \quad  \forall \; x \in \R^{4} \backslash \{0\},     \label{eqn:CA06_1}
\end{align}
where $f=\frac{|x|^2}{4}$. We show that $v$ must be $c|x|^{-2}$ for some constant $c$.  In other words, $v$ is a multiple of the Green function poled at the origin. 
This rigidity theorem is in the flavor of the classical B\^{o}cher's decomposition theorem for harmonic functions(c.f. Theorem 3.9 of~\cite{ABR92}).   
However, the rigidity here is even stronger, due to the ad hoc choice of $f=\frac{|x|^2}{4}$ and the non-zero eigenvalue.
The proof of this rigidity theorem follows the same route as the classical B\^{o}cher's theorem, by using spherical average.
The complete proof is provided in section 3. 

We reduce the proof of Theorem~\ref{thm:1} to the aforementioned rigidity theorem. 
If the Ricci shrinker $(M,p,g,f)$ is very close to a flat cone $\R^{4}/\Gamma$, then $R$ is close to zero function on a very large annulus part $B(0,\delta^{-1}) \backslash B(0, \delta)$.
We rescale the function $R$ by multiplying them with $\alpha^{-1}$, where $\alpha$ is the maximum of $R$ on the unit sphere $\partial B(p, 1)$.  
The rescaled function is denoted by $V$.  Note that the underlying metrics are not changed at all. 
The proof of Theorem~\ref{thm:1} is then carried out by a contradiction argument. 
Suppose Theorem~\ref{thm:1} fails, then we can find a sequence $(M_i, p_i, g_i) \in \mathcal{M}^*(A,H)$ converging to some $\left(\R^{4}/\Gamma,0, g_{E}\right)$, in the pointed-Gromov-Hausdorff topology and hence 
in the pointed-$\hat{C}^{\infty}$-Cheeger-Gromov topology(c.f.~Theorem~\ref{T101} and the discussion below it).  Modulo some a priori estimates from elliptic PDE,  we show that $V_i$ is convergent in proper topology.
Moreover, the limit $V_{\infty}$ is a solution of  (\ref{eqn:CA06_1}) on the smooth part of the limit flat cone and hence force can be lifted to the solution of (\ref{eqn:CA06_1}) on $\R^{4} \backslash \{0\}$. 
Using the rigidity of solutions of (\ref{eqn:CA06_1}),  we obtain that $V_{\infty}= |x|^{-2}$. However, it will violate our Harnack inequality assumption for the scalar curvature. Therefore, we obtain a desired contradiction to establish the proof of Theorem~\ref{thm:1}.

The technical difficulty of the proof of Theorem~\ref{thm:1} locates in the uniform  a priori estimate of $V$.
Actually, one has to introduce several extra auxiliary functions(c.f. Definition~\ref{dfn:MH23_2}) for the purpose of estimating $V$.
All of these auxiliary functions are indicated by the soliton identities arising from (\ref{E100}). 
Therefore,  it is a regularity problem of system of elliptic equations and inequalities to obtain such estimates.  
These estimates are made possible due to the ad hoc structure of this system and the important progress in the study of 4-d Ricci shrinker recently, e.g., the work of Munteanu-Wang~\cite{MW14}. 
One new ingredient for the proof of Theorem~\ref{thm:1} is to separate the rescaling of the curvature and the rescaling of the metric. 
We only need to use the linear structure of the PDE satisfied by the curvatures. Therefore, we are able to keep the metric un-rescaled, but rescale the curvatures as functions to obtain desired linear PDE solution $V_{\infty}$
on the limit space.  In the literature of Ricci flow study, it seems that the rescaling of metrics and curvatures are always done simultaneously. 

We proceed to discuss the proof of Theorem~\ref{thm:2}.   Besides Theorem~\ref{thm:1}, a rigidity theorem(c.f. Theorem~\ref{thm:MH26_2}) of orbifold Ricci shrinkers is needed.
This theorem states that every  orbifold Ricci shrinker with a scalar curvature zero point must be a flat cone.
It can be proved by a standard maximum principle argument.
Based on this rigidity theorem and Theorem~\ref{thm:1}, our Theorem~\ref{thm:2} follows from a contradiction arguments.
For example, if part (a) of Theorem~\ref{thm:2} fails, then we can extract a sequence of Ricci shrinkers converging to an orbifold Ricci shrinker whose scalar curvature at base point is zero.
Consequently, we obtain a sequence of Ricci shrinkers converging to a flat cone $\R^{4}/\Gamma$, which is impossible by Theorem~\ref{thm:1}. This contradiction establishes the
proof of part (a).  The remainder part of Theorem~\ref{thm:2} can be proved similarly, with extra difficulties which can be solved by delicate application of maximum principles on $R$ and $f$.  \\

This paper is organized as follows. In section 2, we review some elementary results of the Ricci flow and the Ricci shrinkers.  We also introduce important auxiliary functions for the study of Ricci shrinkers.
In section 3, we study the rigidity theorems related to flat cones.   We classify all positive solutions of (\ref{eqn:CA06_1}) which is bounded at infinity.
They are nothing but the constant multiples of Green's function on $\R^4$ poled at the origin.
Moreover, we show that any orbifold Ricci shrinker must be flat cone if the scalar curvature equals zero somewhere. 
In section 4, we develop effective estimates for auxiliary functions on the Ricci shrinkers with an almost flat cone annulus.  
This section is the technical core of this paper. 
In section 5, we provide the complete proof of Theorem~\ref{thm:1} and Theorem~\ref{thm:2}.  
Finally, in section 6, we  list some open questions related to our main theorems. \\

{\bf Acknowledgements}:  
Both authors are grateful to professor Haozhao Li for inspiring discussion.  They also thank professor Xiuxiong Chen, Weiyong He and Song Sun for helpful comments. 
Part of this work was done while both authors were visiting AMSS(Academy of Mathematics and Systems Science) in Beijing and USTC(University of Science and Technology of China) in Hefei, during the summer of 2016. 
They wish to thank AMSS and USTC  for their hospitality.

\section{Preliminaries}

On a complete manifold $M^n$, a Ricci flow solution is a family of smooth metrics $g(t)$ satisfying
\begin{align}
   \frac{\partial}{\partial t} g=-2 Rc.  \label{eqn:MH23_3}
\end{align}
The following  pseudo-locality theorem of Perelman is fundamental. 
\begin{thm}[Theorem 10.3 of Perelman~\cite{Pe1}]
\label{T201}
For every $n \ge 2$ there exist $\delta>0$ and $\epsilon_0>0$ depending only on $n$ with the following property. Let $(M^n,g(t)),\, t\in[0,(\epsilon r_0)^2]$, where $\epsilon \in (0,\epsilon_0]$ and $r_0 \in(0,\infty)$, be a complete solution of the Ricci flow with bounded curvature and let $x_0 \in M$ be a point such that
$$
  inj^{-2}(x,0)+|Rm|(x,0) \leq r^{-2}_0 \quad \text{for} \quad x \in B_{g(0)}(x_0,r_0). 
$$
Then we have the interior curvature estimate
\begin{equation}\label{E201}
|\text{Rm}|(x,t)\le (\epsilon_0r_0)^{-2}
\end{equation}
for $x \in M$ such that $d_{g(t)}(x,x_0) \le \epsilon_0r_0$ and $t \in (0,(\epsilon_0r_0)^2]$.

\label{thm:MI02_1}
\end{thm}

Note that it is not stated clearly whether $M$ is a closed manifold in Perelman's original theorem. However, checking the proof carefully, it is clear that the strategy of the proof works for Ricci flows with bounded curvature at each time slice. 
Rigorously, Theorem~\ref{thm:MI02_1} in the noncompact case follows from the combination of Theorem 8.1 of Chau-Tam-Yu~\cite{CTY} and Theorem 3.1 of B.L. Chen~\cite{CBL07}.\\

A Ricci shrinker $(M, g, f)$ is a complete Riemannian manifold $(M,g)$ together with a smooth function $f: M \to \mathbb R$ such that (\ref{E100}) is satisfied. 
By taking the trace of \eqref{E100}, we have
\begin{equation} \label{E100a}
R+\Delta f=\frac{n}{2}.
\end{equation}
For a Ricci shrinker $(M,g,f)$, there exists a solution $g(t)$ of the Ricci flow with $g(0)=g$ such that
\begin{equation}
g(t)=(1-t)\{\phi^t\}^*(g),
\label{E104}
\end{equation}
where $\phi^t$ is the $1$-parameter family of diffeomorphisms generated by $\frac{1}{1-t}\nabla_g f$. 
In particular, a Ricci shrinker can be extended as an ancient Ricci flow solution. 
By Corollary 2.5. of B.L. Chen~\cite{CBL07}, we see that $R \geq 0$ by maximum principle, even without $|Rm|$ bounded condition. 
Moreover, by the evolution equation $(\partial_t-\Delta)R=2|\text{Rc}|^2$ and the strong maximum principle, either $R >0$ everywhere or $R \equiv 0$ and hence Ricci-flat. 
In the latter case,  it is well known that $(M^n,g)$ is isometric to $(\mathbb R^n, g_E)$,  for example, see Theorem 3.3 of Y. Li~\cite{LY16} for a proof.
Therefore, on a non-flat Ricci shrinker, we have
\begin{align}
  R(x)>0, \quad \forall \; x \in M. 
\label{eqn:MI03_2}  
\end{align}

Fix Riemannian manifold $(M,g)$, recall that Perelman's $\mu(g,1)$-functional is defined as the infimum of  $W(\phi,g,1)$ among all positive smooth functions $\phi$ with compact support on $M$ and with normalization condition $\int_{M} \phi^2 dv=1$, 
where 
\begin{align*}
   W(\phi,g,1) \triangleq  \int_{M} \left\{ 4|\nabla \phi|^2 + R\phi^2 -2\phi^2 \log \phi \right\} dv -n-\frac{n}{2} \log (4\pi). 
\end{align*}
In general, for a noncompact manifold $(M, g)$, the minimizer function of $\mu(g,1)$ may not exist.   However, it was proved by Carrillo and Ni that the function $e^{-\frac{f}{2}}$ is always a minimizer of $\mu(g,1)$,
up to adding $f$ by a constant. 

\begin{thm}[Part (i) and (ii) of Theorem 1.1 of Carrillo-Ni~\cite{CN09}]
Suppose $(M, g, f)$ is a Ricci shrinker. 
Then we have
\begin{align}
   \mu(g,1) = W\left(e^{-\frac{f+c}{2}}, g, 1 \right)=-n-\frac{n}{2} \log (4\pi)+\int_{M} \left\{ \left(R+|\nabla f|^2 \right) + f+c\right\} (4\pi)^{-\frac{n}{2}} e^{-f-c} dv,
\label{eqn:MI03_3}   
\end{align}
where $c$ is a constant such that $\int_{M} (4\pi)^{-\frac{n}{2}} e^{-f-c} dv=1$. 
\label{thm:MI03_1}
\end{thm}

By (\ref{eqn:MI03_3}) and the Euler-Lagrangian equation satisfied by minimizer functions, we can easily deduce that $\mu(g,1)=c$.  Therefore, we have the equality
\begin{align}
  \int_{M} (4\pi)^{-\frac{n}{2}} e^{-f} dv= e^{\mu(g,1)}.      \label{eqn:MI03_4}
\end{align}

In this article, we focus on the study of 4d Ricci shrinkers with uniform entropy lower bound. 
Namely, we shall study the Ricci shrinker moduli  $\mathcal M(A,H)$, whose precise definition is stated as follows. 
 
\begin{defn}
Let $\mathcal M(A,H)$ be the family of Ricci shrinkers $(M^4,g,f)$ satisfying
\begin{enumerate}
\item The scalar curvature $R$ of $(M,g)$ is bounded,
\item For any $x,y \in B(p,1)$, $R(x) \le H R(y)$,
\item $\mu(g, 1) \ge -A$.
\end{enumerate}
Let $\mathcal{M}^*(A,H)$ be the collection of all elements in $\mathcal{M}(A,H)$ except the Gaussian soliton $\left(\R^4, g_{E}, e^{-\frac{|x|^2}{4}} \right)$. 
By abusing of notation, we also say $(M,p,g,f) \in \mathcal{M}^*(A,H)$ if $(M,g,f) \in \mathcal{M}^*(A,H)$ and $p$ is a minimum point of $f$ satisfying (\ref{E106}).  
\label{dfn:MH23_1}
\end{defn}
Note that in Definition~\ref{dfn:MH23_1}, it is also required that each Ricci shrinker has bounded scalar curvature.  This is only for technical purpose and could be dropped by further efforts(c.f. Li-Wang~\cite{LW2}). \\

We quote some important estimates from the work of R. Haslhofer and R. M\"{u}ller~\cite{HM11},~\cite{HM15}.  
\begin{lem}[Lemma 2.1 of Haslhofer-M\"{u}ller~\cite{HM11}]
\label{L100}
Let $(M^n,g,f)$ be a Ricci shrinker. Then there exists a point $p \in M$ where $f$ attains its infimum and $f$ satisfies the quadratic growth estimate
\begin{equation}
\frac{1}{4}\left(d(x,p)-5n \right)^2_+ \le f(x) \le \frac{1}{4} \left(d(x,p)+\sqrt{2n} \right)^2
\label{E106}
\end{equation}
for all $x\in M$, where $a_+ :=\max\{0,a\}$.
\end{lem}

\begin{lem}[Lemma 2.2 of Haslhofer-M\"{u}ller~\cite{HM11}]
\label{L101}
There exists a constant $C=C(n)$ such that every Ricci shrinker $(M^n,g,f)$ with $p \in M$ a minimal point of $f$,
\begin{equation}
\text{Vol}\,B(p,r) \le Cr^n.
\label{E106a}
\end{equation}
\end{lem}

\begin{thm}[Theorem 1.1 of Haslhofer-M\"{u}ller~\cite{HM15}]
\label{T101}
Let $(M_i, p_i, g_i,f_i) \in \mathcal M(A,H)$ be a sequence of four dimensional Ricci shrinkers. 
Then by taking subsequence if necessary, we have
\begin{align}
       (M_i, p_i, g_i, f_i) \longright{pointed-\hat{C}^{\infty}-Cheeger-Gromov} (M_{\infty}, p_{\infty}, g_{\infty}, f_{\infty}), 
\label{eqn:MI01_1}       
\end{align}
where $(M_{\infty}, p_{\infty}, g_{\infty},f_{\infty})$ is an orbifold Ricci shrinker with locally finite singular points. 
\end{thm}

Note that the convergence topology in (\ref{eqn:MI01_1}) was stated as ``pointed-orbifold-Cheeger-Gromov" topology.  Let us say a few more words for its precise meaning. 
In fact, (\ref{eqn:MI01_1}) first means that $(M_i,p_i,d_i)$ converges to a length space $(M_{\infty}, p_{\infty}, d_{\infty})$, where $d_i$ is the distance structure induced by $g_i$.  
Then one can decompose the limit space $M_{\infty}$ into regular part $\mathcal{R}(M_{\infty})$ and singular part $\mathcal{S}(M_{\infty})$.
Here regular part $\mathcal{R}(M_{\infty})$ is a smooth manifold equipped with a smooth metric $g_{\infty}$. Locally around each regular point, the metric structure determined by $g_{\infty}$
is identical to $d_{\infty}$.  The singular part is a collection of discrete points.  
The regular part $\mathcal{R}(M_{\infty})$ has an exhaustion $\cup_{j=1}^{\infty} K_j$ by compact sets $K_j$. 
For each compact set $K=K_j$ for some $j$, one can find diffeomorphisms $\varphi_{K,i}$ from $K$ to $\varphi_{K,i}(K)$, a subset of $M_i$ such that
\begin{align*}
    &d_i(\varphi_{K,i}(x), p_i) \to d_{\infty}(x, p_{\infty}), \quad \forall \; x \in K; \\
    &\varphi_{K,i}^*(g_i) \sconv g_{\infty},  \quad \textrm{on} \; K; \\
    &\varphi_{K,i}^*(f_i)  \sconv f_{\infty},  \quad \textrm{on} \; K.
\end{align*}
Although in general the global distance structure induced by $g_{\infty}$ may not be the same as $d_{\infty}$, this difference does not happen whenever $M_{\infty}$ is an orbifold with isolated singularities since
$\mathcal{R}(M_{\infty})$ is geodesic convex.  

Recall that $M_{\infty}$ is an orbifold with discrete singularities. For each singular point, i.e., a point $p \in \mathcal{S}(M_{\infty})$, one can find a small $\delta=\delta(p)$, an open neighborhood $U$ of $p$,  
and a smooth nondegenerate map $\pi: B(0,\delta) \backslash \{0\} \to U \backslash \{p\}$ such that $h_{\infty}=\pi^*(g_{\infty})$ is a smooth metric on $B(0, \delta) \backslash \{0\}$ and $\lim_{x \to 0} h_{\infty}(x)$ exits. 
Moreover, by setting $h_{\infty}(0)=\lim_{x \to 0} h_{\infty}(x)$,  then $h_{\infty}$  is a smooth metric on $B(0,\delta)$. 
Note that $h_{\infty}$ is $\Gamma$-invariant, where $\Gamma$ is the local orbifold group of $p$. 
The triple $(B(0,\delta), \pi, U)$ is called an orbifold chart around $p$, $h_{\infty}$ is called the orbifold lifting of the metric tensor $g_{\infty}$.

By an orbifold Ricci shrinker $(M_{\infty},g_{\infty},f_{\infty})$ we mean that the identity $\text{Rc}_{\infty}+\text{Hess}f_{\infty}=\frac{1}{2}g_{\infty}$ holds smoothly on any orbifold chart after the lifting, 
where $f_{\infty}$ is a smooth function in the orbifold sense. In other words, $f_{\infty}$ is a smooth function in each orbifold chart.
Clearly, the scalar curvature $R_{\infty}$ is also a smooth function in the orbifold sense. By abuse of notation, we use $R_{\infty}(p)$ to denote the value $\pi^*(R_{\infty})(0)$. 

The following beautiful work of O. Munteanu and J.P. Wang is also important for us. 
\begin{thm}(Theorem 2.5 and 2.6 of Munteanu-Wang~\cite{MW14})
Suppose $(M,g,f)$ is a 4-dimensional Ricci shrinker with bounded scalar curvature.  Then there is a constant $L$ depending on $M$ such that
\begin{align}
  \sup_{M} \frac{|Rm|+|\nabla Rm|}{R} < L.
    \label{eqn:MH20_4}    
\end{align}
  \label{thm:MH20_4}
\end{thm}
In particular, (\ref{eqn:MH20_4}) implies that each soliton in $\mathcal{M}^*(A,H)$ has bounded curvature and Theorem~\ref{T201} can be applied.

Now for a general shrinking soliton $(M,g,f)$, as it can be regarded as a  normalized Ricci flow solution, we have the following elliptic equations where $\Delta_f=\Delta-\langle \nabla, \nabla f \rangle$.
\begin{align}
&\Delta_{f} f=\frac{n}{2}-f, \label{eqn:MH23_7}\\
&\Delta_{f} f^{-1}=f^{-1}-2Rf^{-3}+ \left(2-\frac{n}{2} \right) f^{-2}.    \label{eqn:MH22_2}
\end{align}
The particular case of (\ref{eqn:MH22_2}) in dimension four is
\begin{align}
   \Delta_f f^{-1}=f^{-1}-2Rf^{-3}.     \label{eqn:MH22_3}
\end{align}
The following evolution equations are well-known(c.f. for example, Munteanu-Wang~\cite{MW14}):
\begin{align}
&\Delta_f R =R-2|\text{Rc}|^2,  \label{E215}\\
&\Delta_f \frac{|\text{Rc}|^2}{R}=\frac{|\text{Rc}|^2}{R}+2\frac{|\text{Rc}|^4}{R^2} +2\frac{|\nabla R \text{Rc}-\nabla \text{Rc} R|^2}{R^3}-4\frac{\text{Rm}(\text{Rc},\text{Rc})}{R}.  \label{E216}\\
&\Delta_f R_{ij}=R_{ij}-2R_{ikjl}R_{kl}, \label{eqn:MH23_8}\\
&\nabla_k R_{jk}=R_{jk}f_k=\frac{1}{2} \nabla_j R, \label{eqn:MH23_9}
\end{align}

\begin{defn}
Let $\alpha$ be the maximum of $R$ on the unit sphere $\partial B(p, 1)$.
We define auxiliary functions as follows:
\begin{align}
  &U \triangleq  \frac{|Rc|^2}{R}, \quad V \triangleq \frac{R}{\alpha}, \quad Z \triangleq \frac{|Rc \nabla R -R\nabla Rc|^2}{R^3}. \label{eqn:MH23_11}
\end{align}
\label{dfn:MH23_2}
\end{defn}

\begin{lem}
The auxiliary functions satisfy the following elliptic relationships. 
\begin{align}
&\Delta_f V=V-2UV, \label{E2160a}\\
&\Delta_f U \geq (1-4|Rm|)U+2\left(Z+U^2 \right).   \label{E216a}
\end{align}
\label{lma:MI02_2}
\end{lem}

\begin{proof}
The equation (\ref{E2160a}) follows from (\ref{E215}). 
The inequality (\ref{E216a}) follows from (\ref{E216}). 
\end{proof}

\section{Rigidity theorems related to flat cones}

We investigate the positive solution of the equation 
\begin{align}
\Delta_f v=v
\label{A01}
\end{align}
on $\mathbb \{\R^4/\Gamma\} \backslash \{0\}$, where $\Delta_f=\Delta-\langle \nabla f,\nabla \cdot \rangle$ and $f=\frac{r^2}{4}=\frac{|x|^2}{4}$.
By lifting to the orbifold covering, it suffices to study the solution of (\ref{A01}) for the special case $\Gamma=\{1\}$, where $\mathbb \{\R^4/\Gamma\} \backslash \{0\}$ becomes punctured Euclidean space. 
All the results in this section hold for the general dimension $n$, but we will only focus on dimension $4$ for not distracting the readers' attention from the main stream of this paper. 
For simplicity of notation, we use $B(r)$ to denote the ball of radius $r$ in $\R^4$ centered at the origin $0$.

\begin{thm}[\textbf{Rigidity of eigenfunctions}]
Suppose $v>0$ satisfies (\ref{A01}) on $\mathbb \{\R^4/\Gamma\} \backslash \{0\}$ and $v$ is bounded at infinity, i.e., $\displaystyle \limsup_{x \to \infty} v(x) <\infty$. 
Then we have
\begin{align}\label{A04}
v=br^{-2}
\end{align}
for some constant $b$.

\label{AT02}
\end{thm}

We first consider the possible radial solutions. Let $v(x)=h(r)$ be an radial solution of $\Delta_fv=v$. Then from direct computation, we have
\begin{align}\label{A02}
h''(r)+ \left(\frac{3}{r}-\frac{r}{2} \right)h'(r)=h(r)
\end{align}
for $r>0$. \eqref{A02} is a second order linear ODE, the basis consists of $e^{r^2/4}r^{-2}$ and $r^{-2}$. Therefore the general radial solution of \eqref{A01} is $c_1e^{r^2/4}r^{-2}+c_2r^{-2}$ for some constants $c_1,c_2$.

Now for any solution $v$ of \eqref{A01}, we define its spherical average
$$
A[v](r)=\frac{1}{\Omega_3r^3}\int_{S^r}v \,d\sigma
$$
where $\Omega_3$ is the volume of unit $S^3$.
We have the following decomposition lemmas, whose proofs are similar to those in Theorem 3.9 of Axler-Bourdon-Ramey~\cite{ABR92}.

\begin{lem}\label{AL01}
If $\Delta_fv=v$, then $\Delta_f A[v]=A[v]$.
\end{lem}

\begin{proof}
From the change of variable, $A[v](r)=\frac{1}{\Omega_3}\int_{S^3}v(rw) \,dw$, where $dw$ is the volume form on the unit sphere $S^3$.
Therefore,
\begin{align*}
\Delta_f A[v]&=\frac{1}{\Omega_3}\int_{S^3}\left(\frac{d^2}{dr^2}+\left(\frac{3}{r}-\frac{r}{2}\right)\frac{d}{dr} \right)v(rw) \,dw \\ 
&=\frac{1}{\Omega_3}\int_{S^3}\left(\frac{d^2}{dr^2}+\left(\frac{3}{r}-\frac{r}{2}\right)\frac{d}{dr}+\frac{\Delta_{S^3}}{r^2} \right)v(rw) \,dw \\ 
&=\frac{1}{\Omega_3}\int_{S^3}\Delta_f v(rw) \,dw=\frac{1}{\Omega_3}\int_{S^3} v(rw) \,dw=A[v], 
\end{align*}
where the second identity is true since $\int_{S^3}\Delta_{S^3} v(rw) \,dw=0$ and the third identity holds since we have $\Delta=\frac{d^2}{dr^2}+\frac{3}{r}\frac{d}{dr}+\frac{\Delta_{S^3}}{r^2}$ and $\langle \nabla v,\nabla f \rangle=\frac{r}{2}\frac{d}{dr}v$.
\end{proof}

\begin{lem}\label{AL02}
There exists a constant $c>0$ such that for every positive solution $v$ of \eqref{A01} on $B(1)$, we have 
$$
v(x)>cv(y)
$$
for any $0<|x|=|y| \le 1/2$.
\end{lem}
\begin{proof}
When $|x|=|y|=1/2$, the conclusion follows from the standard Harnack inequality for the elliptic operator $\Delta_f-Id$, see \cite[Theorem $8.20$]{GT01}. For $|x|=|y|=a \le 1/2$, we set $\tilde v(x)=v(ax)$, then 
$$
\Delta_{\tilde f}\tilde v=a^2\tilde v
$$ 
where $\tilde f(x)=f(ax)$. Again, we have the Harnack inequality for $\tilde v$ whenever $|x|=|y|=1/2$. The Harnack constant is independent of $a$ as the coefficients of the above elliptic equations are uniformly controlled.
\end{proof}

\begin{lem}\label{AL03}
If $v$ is a positive solution of \eqref{A01} on $B(1) \backslash \{0\}$ such that $v$ tends to a constant as $|x| \to 1$, then
$v=A[v]$.
\end{lem}
\begin{proof}
From lemma \ref{AL02}, there exists a constant $c \in (0,1)$ such that $v-cA[v] >0$ on $B(1/2)$. On the other hand, since $v-cA[v] \to 0$ as $|x| \to 1$, by the strong maximal principle on $B(1) \backslash B(1/2)$, we conclude that $v-cA[v] >0$ on $B(1)$. Similarly, $v-cA[v]>cA[v-cA[v]]=cA[v]-c^2A[v]$ since $v-cA[v]$ satisfies the same condition as $v$. By iteration, we conclude that
$$
v > g^{(m)}(c)A[v]
$$
for any integer $m>0$, where $g(t)=c+t(1-c)$, $g^{(m)}$ is the m-th iteration of $g$. 
Now as $g^{(m)}(c) \to 1$ if $m \to \infty$, we have $v \ge A[v]$. On the other hand, since $A[v-A[v]]=0$, we conclude that $v=A[v]$ on $B(1)$.
\end{proof}

\begin{lem}
Suppose a positive function $v$ satisfies the equation $\Delta_fv=v$ on $B(1)\backslash \{0\}$, then there exist constants $a,b$ and a smooth function $u$ on $B(1)$ such that
\begin{align}\label{A03}
v=u + a e^{r^2/4}r^{-2} + br^{-2}
\end{align}
on $B(1)\backslash \{0\}$.

\label{lma:MH27_1}
\end{lem}

\begin{proof}
We first find a solution $u$ of (\ref{A01}) on $B(1/2)$ such that $u=v$ on $\partial B(1/2)$. Then we consider the function
$$
w=r^{-2}+v-u.
$$
Since $w \to +\infty$ as $|x| \to 0$, by maximum principle, $w$ is positive. Now from lemma \ref{AL03}, $w$ is radial on $B(1/2)$. Therefore on $B(1/2)$, $w=ae^{r^2/4}r^{-2}+(b+1)r^{-2}$ for some constants $a$ and $b$. That is, $v=u+ae^{r^2/4}r^{-2}+br^{-2}$ on $B(1/2)$. Now we can extend $u$ to $B(1)$ by defining $ u=v-a e^{r^2/4}r^{-2}-b r^{-2}$.
\end{proof}

Based on the previous preparation, we are able to finish the proof of Theorem~\ref{AT02} now.

\begin{proof}[Proof of Theorem~\ref{AT02}]
From Lemma~\ref{lma:MH27_1}, we can decompose $v$ as
$$
v=u-a e^{r^2/4}r^{-2}+br^{-2}
$$
on $B(1)\backslash \{0\}$. Note that we can extend $u$ to a solution of $\Delta_f u=u$ on $\mathbb R^4$ by defining 
$$
u=v+ae^{r^2/4}r^{-2}-br^{-2}
$$
outside $B(1)$. In other words, the decomposition holds on $\mathbb R^4 \backslash \{0\}$.

If $a=0$, we conclude that $u$ is a bounded solution of $\Delta_fu=u$.

Now we choose a cutoff function $\phi$ supported in $B(2r)$ which is equal to $1$ on $B(r)$. Moreover, we require that  $|\nabla \phi|\le Cr^{-1}$. 
Multiplying both sides of $\Delta_fu=u$ by $\phi^2u$ and integrating by parts, we have
\begin{align}\label{A05}
\int|\nabla (\phi u)|^2 \,d\mu \le \int |\nabla \phi|^2u^2 \, d\mu
\end{align}
where $d\mu=e^{-r^2/4}dx$.
By our choice of $\phi$, we have
\begin{align}\label{A06}
\int_{B(r)}|\nabla u|^2 \,d\mu &\le \frac{C}{r^2}\int_{B(2r)} u^2 \, d\mu 
\le \frac{C}{r^2}\int_{B(2r)} e^{-|x|^2/4} \, dx
\end{align}
since $u$ is uniformly bounded.
Then it is easy to see that the last term of the above inequality tends to $0$ as $r \to +\infty$.
Therefore,  $u$ must be a constant. As $\Delta_fu=u$, $u$ must be $0$ and hence $v=b r^{-2}$.

Now we consider the other case when $a \ne 0$.
We rewrite $v=u-a(e^{r^2/4}-1)r^{-2}+(b-a)r^{-2}$. It is obvious that $(e^{r^2/4}-1)r^{-2}$ is a smooth function on $\mathbb R^4$, so from the first case $u-a(e^{r^2/4}-1)r^{-2}=0$ 
and $v=(b-a)r^{-2}$. 
\end{proof}

\begin{thm}[\textbf{Scalar rigidity of Ricci shrinkers}]
Let $(M,g,f)$ be a four dimensional orbifold Ricci shrinker such that $f$ has a minimal point and $R=0$ at some point, then $(M,g,f)$ is a flat cone $(\mathbb R^4 / \Gamma, g_E,f_E)$.
\label{thm:MH26_2}
\end{thm}

\begin{lem}\label{lma:MH27_2}
Let $(M,g,f)$ be a four dimensional orbifold Ricci shrinker such that $f$ has a minimal point. If $g$ is Ricci flat, then $(M,g,f)$ is a flat cone $(\mathbb R^4 / \Gamma, g_E,f_E)$.
\end{lem}

\begin{proof}
We denote a minimal point of $f$ by $p$, which may be a singular point. 
Since $g$ is Ricci flat, the soliton equation (\ref{E100}) reads as  $\text{Hess} f=\frac{g}{2}$. 
The identity (\ref{E101}) degenerates as $\left| \nabla \sqrt{f}\right|=\frac{1}{2}$, which yields that $f=\frac{r^2}{4}$ where $r$ is the distance to $p$.
Note that the geodesic convexity of the regular part of $M$ is essentially used.  More details can be found in Theorem 3.3 of  Y. Li~\cite{LY16}.

We claim that there is no other singular point than $p$.  For otherwise, we can find another singular point $q$ such that $r(q)$ is minimal. Now we connect $p$ and $q$ by a minimal geodesic 
$\gamma(t) \,, t \in [0,1]$ such that $\gamma(0)=p$ and $\gamma(1)=q$.  At any point $\gamma(a)$ for $a \in (1/2,1)$, we have
$$
|\nabla f|=\frac{r}{2} \ge C_0
$$
for some $C_0>0$. But this is impossible, since $|\nabla f|(q)=0$ if we lift it to the orbifold chart as $q$ is a singular point. 
Therefore, $p$ is the unique singular point on $M$ as we claimed. 

We proceed to show that $M$ is a metric cone, which is smooth away from $p$. Indeed, from the above arguments we have $\displaystyle  \mathcal{L}_{\nabla f}g=2\text{Hess}_f =g$, which implies that
$$
 \mathcal{L}_{\nabla r}g=\frac{2}{r}(g-dr^2).
$$
Therefore, for any vector fields $U,V$ such that $[U,\partial_r]=[V,\partial_r]=0$, we have
$$
\partial_r(g(U,V))=(\mathcal{L}_{\nabla r}g)(U,V)=\frac{2}{r}g(U,V).
$$
Now it is immediate that $g=dr^2+r^2\tilde g$ where $\tilde g$ is a smooth metric on a closed 3-manifold $\Sigma$ defined by $r=\sqrt{4f}=1$. As $g$ is Ricci flat, direct computation shows that $(\Sigma^3, \tilde g)$ is Einstein with Einstein constant $3$, 
which must be space form of constant sectional curvature $1$.  Therefore, $(\Sigma^3, \tilde{g})$ is isometric to $\left(S^3 / \Gamma, g_{S} \right)$ for some finite subgroup $\Gamma \subset O(4)$ acting freely on $S^3$.
Consequently,  $(M,g,f)$ is nothing but $(\mathbb R^4 / \Gamma, g_E,f_E)$. 
\end{proof}

Although it is not needed in our proof, we remark that the requirement of isolated singularity in Lemma~\ref{lma:MH27_2} can be replaced by much weaker conditions, e.g., the singularity is codimension 4 and the regular part is geodesic convex.
This can be proved following part of the argument in Theorem 4.18 of Chen-Wang~\cite{CW5}. 

Now we are ready to finish the proof of Theorem~\ref{thm:MH26_2}. 

\begin{proof}[Proof of Theorem~\ref{thm:MH26_2}:]
 Suppose $R(q)=0$. No matter whether $q$ is a smooth point, we can find an orbifold chart where $\Delta_{\tilde{f}} \tilde{R}=\tilde{R} + 2|\tilde{Rc}|^2$, where $\tilde{}$ means the corresponding functions lifted to the orbifold chart.  
 By strong maximum principle, we obtain $\tilde{R} \equiv 0$ in the chart and consequently $R \equiv 0$ in a small neighborhood of $q$.  Then we apply strong
 maximum principle on $\Delta_f R=R+2|Rc|^2$ and obtain that $|Rc| \equiv 0$ on the regular part of $M$.   Therefore, $(M,g,f)$ is the flat cone by Lemma~\ref{lma:MH27_2}. 
\end{proof}

\section{Ricci shrinker with an almost flat cone annulus}
\label{sec:annulus}

Our object is to study the pointed Ricci shrinkers $(M,p,g,f)$ very close to the flat cone $\R^{4}/\Gamma$ in the pointed-Gromov-Hausdorff topology. 
However, in light of Theorem~\ref{T101},  such shrinkers must be nearby $\R^{4}/\Gamma$ in the pointed-$\hat{C}^{\infty}$-Cheeger-Gromov topology(c.f. (\ref{eqn:MI01_1})).
From its definition, it is clear that the level set annulus part $\left\{ x \left| \delta \leq \sqrt{4f(x)} \leq \delta^{-1}  \right.\right\}$ must be very close to the standard annulus on the flat cone.
Motivated by this observation, we provide the following definition. 

\begin{defn}
  We say the pointed Ricci shrinker $(M,p,g,f)$ has an almost flat cone annulus $\Omega^{\delta, \delta^{-1}}$ with respect to the flat cone $\R^4/\Gamma$,  if there exists a diffeomorphism
  \begin{align*}
     \varphi:  B\left(0, \delta^{-1} \right) \backslash B(0, \delta)  \mapsto \Omega^{\delta, \delta^{-1}}=\varphi \left(  B\left(0, \delta^{-1} \right) \backslash B(0, \delta) \right) \subset M
  \end{align*}
  such that the following estimates hold: 
  \begin{itemize}
  \item[(a).] $\left| d(p, \varphi(x))-|x| \right|<0.1 \delta$ for every $x \in B\left(0, \delta^{-1} \right) \backslash B(0, \delta)$, where $|x|=d(x,0)$.   
  \item[(b).] $B(p,1)\backslash B(p,2\delta) \subset \Omega^{\delta,\delta^{-1}}$.
  \item[(c).] $\norm{\varphi^* g - g_{E}}{C^{5}(B\left(0, \delta^{-1} \right) \backslash B(0, \delta))}<\delta$. 
  \item[(d).] $\norm{\varphi^* f - \frac{|x|^2}{4}}{C^{5}(B\left(0, \delta^{-1} \right) \backslash B(0, \delta))}<0.01\delta^2$. 
  \end{itemize}
\label{dfn:MI01_1}  
\end{defn}

From Definition~\ref{dfn:MI01_1}, it is clear that $\Omega^{\delta,\delta^{-1}}$ is very close to the set $\left\{ x \left|\delta \leq \sqrt{4f(x)} \leq \delta^{-1}  \right. \right\}$.
Moreover, $\Omega^{\delta, \delta^{-1}}$ has the advantage of being diffeomorphic to $B(0,\delta^{-1}) \backslash B(0, \delta)$, a standard annulus in
the flat cone $\R^{4} /\Gamma$.  Therefore, we can do analysis on $B(0,\delta^{-1}) \backslash B(0, \delta)$, with respect to the pull back metric $\varphi^*(g)$, which is
very close to the flat metric.  One can see Figure~\ref{fig:coneannulus} for intuition. 
Note the function $\left|\varphi^{-1} \right|$ is very close to $f$. In particular, we have 
\begin{align*}
    \left\langle \nabla \left|\varphi^{-1} \right|, \nabla f \right \rangle> 0.9 |\nabla f|^2>0, \quad \forall \; x \in \Omega^{\delta, \delta^{-1}}. 
\end{align*}
Therefore $M \backslash \Omega^{\delta, \delta^{-1}}$ contains two parts which are disconnected to each other.
One of them has large value of $f$, say $f>0.1 \delta^{-2}$. This part is called the outer part. 
The other one is the part  with small value of $f$, say $f<10 \delta^2$. We call this part as inner part. 
For simplicity of notation, for each $r_0 \in (\delta, \delta^{-1})$, we denote the union of $\Omega^{r_0, \delta^{-1}}$ and the outer part by $\Omega^{r_0+}$. 
In other words, we have
\begin{align}
   \Omega^{r_0+} \triangleq \Omega^{r_0,\delta^{-1}} \cup \{\textrm{outer part}\}=\varphi \left\{ B(0, \delta^{-1}) \backslash B(0, r_0) \right\} \cup \{\textrm{outer part}\}.   \label{eqn:MI03_1}
\end{align}
Note that we use $\Omega^{r_1,r_2}$ to denote $\varphi \left\{ B(0, r_2) \backslash B(0, r_1) \right\}$  whenever $\delta \leq r_1 <r_2 \leq \delta^{-1}$.

 \begin{figure}
 \begin{center}
 \psfrag{A}[c][c]{\color{red}$B\left(0,\delta^{-1} \right) \backslash B(0,\delta)$}
 \psfrag{B}[c][c]{\color{red}$\Omega^{\delta,\delta^{-1}}$}
 \psfrag{C}[c][c]{\color{red} $\varphi$}
 \psfrag{D}[c][c]{$M$}
 \psfrag{E}[c][c]{$\R^4/\Gamma$}
 \psfrag{F}[c][c]{$p$}
 \psfrag{G}[c][c]{$0$}
 \includegraphics[width=0.5 \columnwidth]{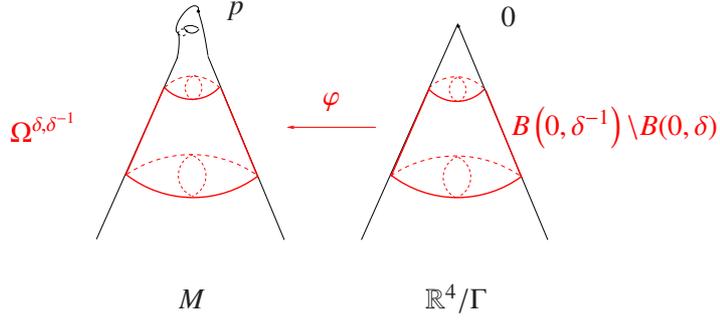}
 \caption{Almost flat cone annulus}
 \label{fig:coneannulus}
 \end{center}
 \end{figure}

A Ricci shrinker with an almost flat cone annulus has many special properties.  For example, $|Rm| \leq \frac{1}{4}$ on  $\Omega^{\delta, \delta^{-1}}$. 
Consequently,  (\ref{E216a}) becomes
\begin{align}
  \Delta_f U \geq 2\left(Z+U^2\right) \geq 0, \quad \textrm{on} \;   \Omega^{\delta, \delta^{-1}}.  
\label{eqn:MH21_5}
\end{align}
In fact, the Ricci shrinker equation (\ref{E100}) is very rigid.  Much more global properties of $(M,g,f)$ can be shown.

\begin{lem}
 Suppose $(M, g, f) \in \mathcal{M}^{*}(A)$ has an almost flat cone annulus $\Omega^{\delta, \delta^{-1}}$.  
 Then $M$ is noncompact.   Moreover, $\Omega^{1+}$ is diffeomorphic to $S^3/\Gamma \times [1,\infty)$ and there  is a uniform $C=C(A)$ such that
    \begin{align}
     |Rm|(y) \leq C d^{-2}(y,p), \quad \forall \; y \in M \backslash B(p,1). 
     \label{eqn:MH20_5}
     \end{align}
 In other words, the curvature is uniformly quadratically decaying at infinity.    
\label{lma:GH20_1}
\end{lem}

\begin{proof}

Running Ricci flow from the Ricci shrinker $(M, g, f)$, we obtain a family of smooth metrics $g(t)$ satisfying (\ref{E104}).  
Up to a rescaling argument, one can apply Perelman's pseudo-locality theorem, i.e., Theorem~\ref{thm:MI02_1},  to obtain
\begin{align}
|\text{Rm}|(q,t)\le 3,   \quad \forall \; q \in \Omega^{r_0, 2r_0},  \; t \in (0,1),    \label{E203}
\end{align}
where we choose $1<r_0<\delta^{-1}$ large enough so that Theorem~\ref{thm:MI02_1} can be applied. 
However,  by considering the Ricci flow solution of the Ricci shrinker \eqref{E104}, the above inequality means that
\begin{align}
|\text{Rm}|(\phi^t(q),0)\le 3(1-t),  \quad \forall \; q \in \Omega^{r_0, 2r_0},  \; t \in (0,1).
\label{E206}
\end{align}
In particular, for each $q \in \Omega^{r_0, 2r_0}$ and $t \in (0,1)$, we have the scalar curvature bound
\begin{align}
   R(\phi^t(q),0) \leq 3n(n-1)(1-t) \leq 3n(n-1)<0.01 r_0^2.   \label{eqn:MI02_0}
\end{align}
However, we have $f>0.1 r_0^2$ on $\Omega^{r_0+}$.  In light of (\ref{E101}), we know that 
\begin{align}
 |\nabla f|^2(\phi^t(q))>0.09r_0^2,  
\label{eqn:MI02_3}
\end{align}
whenever $\phi^{t}(q) \in \Omega^{r_0+}$. 
Therefore, along the flow line of $\phi^t(q)$, there is no critical point of $f$ and $f(\phi^t(q))$ is an increasing function of $f$ since
\begin{align}
    \frac{d}{dt}f(\phi^t(q))=\left\langle \nabla f, \frac{d}{dt} \phi^t(q) \right\rangle=\frac{|\nabla f|^2(\phi^t(q))}{1-t}>0.   \label{eqn:MI02_1}
\end{align}
This forces that $\phi^t(q)$ will keep stay in $\Omega^{r_0+}$ whenever
it enters $\Omega^{r_0+}$ at some time $t \geq 0$.  Moreover, the flow line $\phi^t(q), \, t\in [0,1)$ has no stationary point.  Plugging (\ref{E101}) into (\ref{eqn:MI02_1}),  using (\ref{eqn:MI02_0}), we obtain
\begin{align}
    \frac{d}{dt}f(\phi^t(q))=\frac{f-R}{1-t}>\frac{0.9f}{1-t}. 
\label{eqn:MI02_2}
\end{align}
In particular, we have $\displaystyle \lim_{t \to 1^{-}} f(\phi^t(q))=\infty$ and $\displaystyle \sup_{M} f=\infty$. 
Consequently, $M$ is a noncompact manifold since $f$ is a smooth function. 

By (\ref{eqn:MI02_3}), it is clear that $\phi$ induces a diffeomorphism from $\partial \Omega_{2r_0} \times [0, 1)$ to $\Omega_{2r_0+}$ by
\begin{align*}
  \phi:  \partial \Omega_{2r_0} \times [0, 1) \mapsto \Omega_{2r_0+}, \quad (q,t) \mapsto \phi^t(q). 
\end{align*}
Since $\partial \Omega^{2r_0}$ is diffeomorphic to $S^3 /\Gamma$, $[0,1)$ is diffeomorphic to $[2r_0, \infty)$, we see that $\Omega^{2r_0+}$ is diffeomorphic to $S^3 /\Gamma \times [2r_0, \infty)$.
Concatenating this diffeomorphism with the natural diffeomorphism between $\Omega^{1,2r_0}$ and $S^3 /\Gamma \times [1, 2r_0]$, we obtain a diffeomorphism between $\Omega^{1+}$ and $S^3 /\Gamma \times [1, \infty)$.

We continue to show (\ref{eqn:MH20_5}). 
Let $r$ be the distance function to $p$.
Fixing a point $q \in \Omega^{r_0, 2r_0}$,  we have
\begin{align}
\frac{d}{dt}r(\phi^t(q)) &=\left\langle \nabla r,\frac{d}{dt}\phi^{t}(q)  \right\rangle=\frac{1}{1-t}\langle \nabla r, \nabla f \rangle \notag 
\le \frac{1}{1-t} |\nabla f| \notag \le \frac{\sqrt{f}}{1-t}, 
\label{E210}
\end{align}
where the last inequality follows from \eqref{E101}. By \eqref{E106}, the inequality above becomes
\begin{equation}\label{E211}
\frac{d}{dt}r(\phi^t(q)) \le \frac{r(\phi^t(q))+\sqrt{8}}{2(1-t)}.
\end{equation}
Integrating \eqref{E211} yields that
\begin{equation}\label{E212}
r(\phi^t(q)) \le \frac{C_1r_0}{\sqrt{1-t}}
\end{equation}
for some positive constant $C_1=C_1(A)$ independent of $q$.  
Therefore, (\ref{eqn:MH20_5}) follows from the combination of  \eqref{E206} and \eqref{E212}. 
\end{proof}

\begin{lem}
 Suppose $(M, g, f) \in \mathcal{M}^{*}(A)$ has an almost flat cone annulus $\Omega^{\delta, \delta^{-1}}$.   
 Then we have the following properties. 
 \begin{itemize}
 \item[(a).] With respect to the measure $d\mu=e^{-f}dv$, we have
   \begin{align}
    \int_{M} |Rm|^2 d\mu < C'
 \label{eqn:GH20_1}   
 \end{align}
 for some $C'=C'(A)$. 
 \item[(b).] At the infinity end of $M$, we have
\begin{align}
  \lim_{y \to \infty} \frac{|Rm|^2}{R}(y)=0.
  \label{eqn:MH20_6}
\end{align}
 \end{itemize}
\label{lma:MI02_3}
\end{lem}

\begin{proof}
 The inequality (\ref{eqn:GH20_1}) follows from the combination of the uniform lower bound of $f$ in (\ref{E106}) and the volume ratio upper bound in (\ref{E106a}). 
 
 The equation (\ref{eqn:MH20_6}) follows from the combination of Munteanu-Wang's inequality (\ref{eqn:MH20_4}) and  the quadratic curvature decay estimate (\ref{eqn:MH20_5}). 
\end{proof}

\begin{lem}
   There is a uniform $C=C(A)$ such that
   \begin{align}
       \int_{M} \left(Z+U^2 \right) d\mu \leq C.   \label{eqn:MH21_6}
   \end{align}
\label{lma:GH20_2}   
\end{lem}

\begin{proof}
  
  Choose $\lambda$ very large satisfying $|\partial B(p, \lambda)| \leq C\lambda^3$ for some $C$ independent of $M$. 
  Since both $U$ and $Z$ are nonnegative,  integrating on $B(p,\lambda)$ implies that
  \begin{align*}
    \int_{B(p, \lambda)} \left(\Delta_{f} U \right) e^{-f}dv \geq \int_{B(p,\lambda)} \left( 2\left(Z+U^2 \right)-4|Rm|U \right) e^{-f}dv
    \geq \int_{B(p,\lambda)}\left( Z+U^2-4|Rm|^2 \right) e^{-f}dv.
  \end{align*}
  It follows that
  \begin{align}
    \int_{B(p,\lambda)}\left( Z+U^2 \right) e^{-f} dv &\leq 4 \int_{B(p,\lambda)} |Rm|^2 e^{-f}dv + \int_{B(p, \lambda)} \left(\Delta_{f} U \right) e^{-f}dv \notag\\
    &\leq 4 \int_{B(p,\lambda)} |Rm|^2 e^{-f}dv +\int_{\partial B(p, \lambda)} |\nabla U| e^{-f} d\sigma \notag\\
    &\leq C + \int_{\partial B(p, \lambda)} |\nabla U| e^{-f} d\sigma.  \label{eqn:MH20_3} 
  \end{align}
  Recall that $U=\frac{|Rc|^2}{R}$. Therefore, we have
  \begin{align*}
    |\nabla U|=\left|\nabla \frac{|Rc|^2}{R} \right|=\left| \frac{R|Rc| \nabla |Rc|-|Rc|^2 \nabla R}{R^2}\right| \leq C \frac{|Rc|^2|\nabla Rm|}{R^2}=CU \frac{|\nabla Rm|}{R} \leq CLU, 
  \end{align*}
  where we used Theorem~\ref{thm:MH20_4} in the last step.  
  Since $U=\frac{|Rc|^2}{R} \leq C\frac{|Rm|^2}{R}$, by (\ref{eqn:MH20_6}), we know $U$ is a bounded function on $M$.   Therefore, we have
  \begin{align*}
     \int_{\partial B(p, \lambda)} |\nabla U| e^{-f} d\sigma \leq L e^{-\frac{\lambda^2}{8}} \lambda^3
  \end{align*}
  for some constant $L$ depending on $M$ but independent of $\lambda$. 
  In light of (\ref{E106a}), we can choose a sequence of $\lambda_j \to \infty$ such that $|\partial B(p, \lambda_j)| \leq C \lambda_j^3$.  
  Plugging the above inequality into (\ref{eqn:MH20_3})  and letting $\lambda_j \to \infty$,  we obtain (\ref{eqn:MH21_6}).   
\end{proof}

\begin{lem}
  There exists a uniform constant $C=C(A)$ such that
  \begin{align}
    &\sup_{\Omega^{0.5r,r}} U \leq Cr^{-2} \left\{ \int_{\Omega^{0.25r, 2r}} U^2 d\mu \right\}^{\frac{1}{2}},     \label{eqn:MH20_11}\\
    &\int_{\Omega^{r+}} |\nabla U|^2 d\mu \leq C r^{-2} \int_{\Omega^{0.5r,r}} U^2 d\mu, \label{eqn:MH20_12}
  \end{align}
  for every $r \in (2\delta, 1)$. 
  \label{lma:MH20_4}
\end{lem}

\begin{proof}

 We first prove (\ref{eqn:MH20_11}).   Note that $|Rm| \leq \frac{1}{4}$ on $\Omega^{0.25r,2r}$. 
 Then we have
  \begin{align}
    \Delta_{f} U \geq U+2U^2-4|Rm|U \geq 2U^2 \geq 0.  \label{eqn:MH20_10}
  \end{align}
 Then  (\ref{eqn:MH20_11}) follows from standard Moser iteration.

 We then focus on the proof of (\ref{eqn:MH20_12}).  
 Let $\psi$ be a cutoff function supported on $\Omega^{0.5r+}$ and equals $1$ on $\Omega^{r+}$.   Moreover, $|\nabla \psi| \leq 10r^{-1}$. 
  Similar to (\ref{eqn:MH20_10}), it is clear that $\Delta_{f} U \geq 0$ on $\Omega^{r+}$.
  It follows from integration by parts that
  \begin{align*}
    -\int_{M} \psi^2 |\nabla U|^2 d\mu - \int_{M} 2\psi U \langle \nabla \psi, \nabla U\rangle d\mu=\int_{M} (\psi^2 U) \Delta_{f} U d\mu \geq 0, 
  \end{align*}
  which implies that
  \begin{align*}
    \int_{M} \psi^2 |\nabla U|^2 d\mu \leq - \int_{M} 2\psi U \langle \nabla \psi, \nabla U\rangle d\mu \leq \frac{1}{2} \int_{M} \psi^2 |\nabla U|^2 d\mu + 2\int_{M} U^2|\nabla \psi|^2d\mu. 
  \end{align*}
  Note that $|\nabla \psi| \leq Cr^{-1}$.  Hence we arrive
  \begin{align*}
    \int_{\Omega^{r+}} |\nabla U|^2 d\mu \leq \int_{M} \psi^2 |\nabla U|^2 d\mu \leq Cr^{-2} \int_{supp(\nabla \psi)} U^2 d\mu \leq Cr^{-2} \int_{\Omega^{0.5r,r}} U^2 d\mu.
  \end{align*}
  The proof of (\ref{eqn:MH20_12}) is complete. 

\end{proof}

\begin{prop}[\textbf{Estimate of $U$ and $Z$}]
 For each $\rho \in \left(16 \delta, 1 \right)$, we have the estimates
 \begin{align}
      &\int_{\Omega^{\rho+}} \left(Z+U^2\right) d\mu \leq C  \left\{ \log \frac{\rho}{\delta}\right\}^{-\frac{1}{2}}, \label{eqn:MH21_7}\\
      &\int_{\Omega^{\rho+}} |\nabla U|^2 d\mu \leq  C \rho^{-2} \left\{ \log \frac{\rho}{\delta}\right\}^{-\frac{1}{2}}, \label{eqn:MH21_8}\\
      &\sup_{\Omega^{\rho, \rho^{-1}}} U \leq   C \rho^{-2} \left\{ \log \frac{\rho}{\delta}\right\}^{-\frac{1}{4}},  \label{eqn:MH21_4}
 \end{align}
 for some uniform constant $C=C(A)$. 
\label{prn:MH21_1}  
\end{prop}

\begin{proof}
   Let $\psi$ be a cutoff function supported on $\Omega^{0.5r+}$ and equal $1$ on $\Omega^{r+}$.  Multiplying $\psi$ to both sides of the inequality
   $\Delta_f U \geq 2(Z+U^2)$ and doing integration by parts, we obtain
   \begin{align*}
    2\int_{M} \left(Z+U^2\right)  \psi d\mu &\leq -\int_{M} \langle \nabla U, \nabla \psi \rangle d\mu \leq Cr^{-1} \int_{\Omega^{0.5r,r}} |\nabla U| d\mu\\
        &\leq Cr^{-1} \left|\Omega^{0.5r,r} \right|^{\frac{1}{2}} \left\{ \int_{\Omega^{0.5r,r}} |\nabla U|^2 d\mu \right\}^{\frac{1}{2}}\\
        &\leq Cr\left\{ \int_{\Omega^{0.5r,r}} |\nabla U|^2 d\mu \right\}^{\frac{1}{2}}, 
   \end{align*}
   where we used the fact that $\Omega^{\delta, \delta^{-1}}$ is cone-like and $\delta<r<1$.   Since $\psi$ is supported on $\Omega^{0.5r+}$ and $\psi \geq 0,  U \geq 0$ always, we arrive at
   \begin{align}
    \int_{\Omega^{r}} \left(Z+U^2 \right) d\mu \leq Cr\left\{ \int_{\Omega^{0.5r,r}} |\nabla U|^2 d\mu \right\}^{\frac{1}{2}}.  \label{eqn:MH21_1}
   \end{align}
   Note that (\ref{eqn:MH20_12}) implies that
    \begin{align}
      \int_{\Omega^{0.5r,r}} |\nabla U|^2 d\mu \leq \int_{\Omega^{0.5r+}} |\nabla U|^2 d\mu \leq C r^{-2} \int_{\Omega^{0.25r,0.5r}} U^2 d\mu 
      \leq C r^{-2} \int_{\Omega^{0.25r,0.5r}} \left(Z+U^2 \right) d\mu.   \label{eqn:MH21_2}
    \end{align}
   Combining (\ref{eqn:MH21_1}) and (\ref{eqn:MH21_2}) yields that
   \begin{align}
     \left\{ \int_{\Omega^{r}} \left(Z+U^2\right) d\mu \right\}^2 \leq C  \int_{\Omega^{0.25r,0.5r}} \left(Z+U^2\right) d\mu.   \label{eqn:MH21_3}
   \end{align}
   We remind the reader that $C$ above depends only on $A$ and does not depend on the manifold $M$. 
   Fix $\rho \in (\delta, 1)$. 
   For each positive integer $i$, let $r_i=2^{-i+1} \rho$.  Then we have
   \begin{align*}
     \left\{ \int_{\Omega^{\rho+}} \left(Z+U^2\right) d\mu \right\}^2&=\left\{ \int_{\Omega^{r_1+}} \left(Z+U^2\right) d\mu \right\}^2 \leq \left\{ \int_{\Omega^{r_i+}} \left(Z+U^2\right) d\mu \right\}^2\\
      &\leq \int_{\Omega^{0.25r_i,0.5r_i}} \left(Z+U^2\right) d\mu=\int_{\Omega^{2^{-i-1}\rho, 2^{-i}\rho}} \left(Z+U^2\right) d\mu.
   \end{align*}
   In the above inequalities, let $i$ run from $1$ to $k$ and then sum them together, we obtain
   \begin{align*}
     k\left\{ \int_{\Omega^{\rho+}} \left(Z+U^2\right) d\mu \right\}^2 &\leq \sum_{i=1}^{k}\int_{\Omega^{2^{-i-1}\rho, 2^{-i}\rho}} \left(Z+U^2\right) d\mu =\int_{\Omega^{2^{-k-1}\rho, 2^{-1}\rho}} \left(Z+U^2\right) d\mu\\ 
      &\leq \int_{M} \left(Z+U^2\right) d\mu \leq C. 
   \end{align*}
   Consequently, we obtain $\int_{\Omega^{\rho+}} \left(Z+U^2\right) d\mu \leq C k^{-\frac{1}{2}}$, which together with (\ref{eqn:MH20_11}) and (\ref{eqn:MH20_12}) yields that
   \begin{align*}
   &\int_{\Omega^{\rho+}} |\nabla U|^2 d\mu \leq C \rho^{-2} \int_{\Omega^{0.5\rho,\rho}} U^2 d\mu \leq C \rho^{-2} \int_{\Omega^{0.5\rho+}} U^2 d\mu \leq C\rho^{-2} k^{-\frac{1}{2}},\\
   &\sup_{\Omega^{\rho, \rho^{-1}}} U \leq C\rho^{-2} \left\{ \int_{\Omega^{0.5\rho, 4\rho^{-1}}} U^2 d\mu \right\}^{\frac{1}{2}} 
      \leq C\rho^{-2} \left\{ \int_{\Omega^{0.5\rho+}} U^2 d\mu \right\}^{\frac{1}{2}}  \leq C\rho^{-2} k^{-\frac{1}{4}}. 
   \end{align*}
   Therefore, (\ref{eqn:MH21_7}), (\ref{eqn:MH21_8}) and (\ref{eqn:MH21_4}) follow from the above inequalities by setting $k \sim \log_{2} \frac{\rho}{\delta}$. 
\end{proof}

\begin{prop}[\textbf{Estimate of $V$}]
 For each $\rho \in (100\delta, 0.01 \delta^{-1})$, we have
\begin{itemize}
\item[(a).] $V$ satisfies uniform Harnack inequality:
     \begin{align}
       C_{\rho}^{-1} \leq V(x) \leq C_{\rho}, \quad \forall \; x \in \Omega^{\rho, \rho^{-1}}. 
      \label{eqn:MH22_7} 
     \end{align}
\item[(b).] $V$ has uniformly bounded $C^{1,\frac{1}{2}}$-norm:
     \begin{align}
       \norm{V}{C^{1,\frac{1}{2}}\left(\Omega^{\rho, \rho^{-1}} \right)} \leq C_{\rho}. 
       \label{eqn:MH22_8}
     \end{align}
\item[(c).] There exists a uniform constant $C=C(A)$ such that
 \begin{align}
    \sup_{x \in \Omega^{1+}} V(x) \leq C.   \label{eqn:MH22_1}
 \end{align}     
\end{itemize}
\label{prn:MH21_2} 
\end{prop}

\begin{proof}

Part (a) follows from standard elliptic equation theory(c.f. Theorem 8.20 of the classical book Gilbarg-Trudinger~\cite{GT01}), since $\Delta_f V=(1-U)V$ on $\Omega^{0.5\rho, 2\rho^{-1}}$ and $U$ is very small 
by inequality (\ref{eqn:MH21_4}) of Proposition~\ref{prn:MH21_1}. 
Since $(1-U)V$ is uniformly bounded on $\Omega^{0.5\rho, 2\rho^{-1}}$, it follows from the uniform ellipticity of the operator $\Delta_f$ that $\norm{V}{W^{2,p}(\Omega^{0.75\rho, 1.5\rho^{-1}})}$ is uniformly bounded
by $C_{\rho, p}$ for each $p>n=4$.  Then inequality (\ref{eqn:MH22_8}) in part (b) follows from Sobolev embedding theorem. 

We now focus on the proof of part (c). Recall that $ \Delta_f f^{-1}=f^{-1}-2Rf^{-3}$. 
  Choose $c_0$ such that $V-c_0 f^{-1}$ is negative in $\partial \Omega^{1+}$.  Then we have
  \begin{align*}
       \Delta_{f} \left( V-c_0 f^{-1}\right)= V-2UV-c_0 \left( f^{-1}-2Rf^{-3}\right) \geq \frac{1}{2}V-c_0 f^{-1}. 
  \end{align*}
  Note that $V-c_0f^{-1}$ approaches $0$ at infinity of $M$. If $V-c_0f^{-1} \leq 0$, we obtain upper bound of $V$ directly.  Otherwise, $V-c_0f^{-1}$ must achieve some positive maximum at some point $y \in \Omega^{1+}$,
  where we have
  \begin{align*}
    0 \geq \frac{1}{2}V(y) -c_0 f^{-1}(y) \quad \Rightarrow \quad V(y) \leq 2c_0 f^{-1}(y). 
  \end{align*}  
  Therefore, for arbitrary point $x \in \Omega^{1+}$, we have
  \begin{align*}
   &\quad \; V(x)-c_0f^{-1}(x) \leq V(y)-c_0f^{-1}(y) \leq c_0 f^{-1}(y), \\
   &\Rightarrow V(x) \leq  c_0 \left( f^{-1}(x) + f^{-1}(y) \right) \leq C,
  \end{align*}
  since both $x,y$ locate in $\Omega^{1+}$ where $f$ is uniformly bounded from below by a positive number. 
  The proof of part (c) is complete. 
\end{proof}

\section{Proof of  Main theorems}

The proof  of Theorem~\ref{thm:1} is carried out by a contradiction argument.  The basic idea is to obtain a limit flat cone, together with a eigenfunctions $V_{\infty}$ on
the regular part of the flat cone, whenever the statement of Theorem~\ref{thm:1} fails.  Checking the formation of $V_{\infty}$, we shall show that it must be of the form $cr^{-2}$ by rigidity of the eigenfunctions in Theorem~\ref{AT02}.  However, this will contradicts our Harnack inequality.

\begin{proof}[Proof of Theorem \ref{thm:1}]
Suppose Theorem~\ref{thm:1} fails, then we can have a sequence of Ricci shrinkers $(M_i, g_i, f_i) \in \mathcal{M}^{*}(A,H)$ such that
\begin{align*}
   d_{PGH} \left( (M_i, p_i, g_i), \left(\R^{4}/\Gamma, 0, g_{E} \right)\right) \to 0. 
\end{align*}
By Theorem~\ref{T101},  the convergence topology can be improved as follows
\begin{align*}
    (M_i, p_i, g_i)  \longright{pointed-\hat{C}^{\infty}-Cheeger-Gromov} \left(\R^{4}/\Gamma, 0, g_{E}\right). 
\end{align*}
Therefore, for each fixed small $\delta$,  the manifold $M_i$ has an approximating annulus part $\Omega_i^{\delta, \delta^{-1}}$ to the standard annulus $\Omega_{\infty}^{\delta, \delta^{-1}}$ in the flat cone $\R^{4}/\Gamma$. 
Then the discussion in Section~\ref{sec:annulus} can be applied.  We have uniform estimates of $V_i$ on $\Omega_i^{\rho, \rho^{-1}}$ by Proposition~\ref{prn:MH21_2},
where $\rho=\delta^{\frac{1}{1000}}$.  Let $i \to \infty$, we obtain function $V_{\infty}$ on $\Omega_{\infty}^{\rho, \rho^{-1}}$. Moreover, $V_{\infty}$ is a $C^{1,\frac{1}{2}}$-function. The convergence from $V_i$ to $V$ happens in $C^{1, \frac{1}{3}}$-topology.
In particular, we can find a point $z_{\infty} \in \partial B(0,1)$ such that
\begin{align}
    V_{\infty}(z_{\infty})=1,   \label{eqn:MH23_1}
\end{align}
according to the choice of the normalization condition for $V_i$. 
Let $\delta \to 0$, we obtain $\rho \to 0$.  Consequently, we have a function $V_{\infty}$ defined on $\left\{\R^{4}/\Gamma \right\} \backslash \{0\}$.
Moreover, by estimate (\ref{eqn:MH21_4}) in Proposition~\ref{prn:MH21_1},  taking limit of the first equation of (\ref{E216a}) implies that
\begin{align}
   \Delta_{f_{\infty}} V_{\infty}=V_{\infty}    \label{eqn:MH23_2}
\end{align}
in the distribution sense on $\left\{\R^{4}/\Gamma \right\} \backslash \{0\}$.   Note that $f_{\infty}=\frac{1}{4}r^2$, where $r$ is the distance to the vertex. Clearly, $f_{\infty}$ is smooth.
By standard elliptic theory, we know that $V_{\infty}$ is a smooth function and (\ref{eqn:MH23_2}) holds in the classical sense. 
In light of part (c) of Proposition~\ref{prn:MH21_2},  $V_{\infty}$ is bounded outside the unit ball. 
It follows from Theorem~\ref{AT02} and (\ref{eqn:MH23_1}) that $V_{\infty}=r^{-2}$. Therefore, when $i$ is sufficiently large, the fast increasing rate of $R_i$ near $p$ will violate our uniform Harnack constant $H$.
The proof  of Theorem~\ref{thm:1} is complete.
\end{proof}

The proof of Theorem~\ref{thm:2} is based on maximum principle and the application of Theorem~\ref{thm:1}. 
Roughly speaking, if the statements in Theorem~\ref{thm:2} fail, then we can obtain a sequence of Ricci shrinkers converging to a flat cone which is impossible by Theorem~\ref{thm:1}.
To force the limit to be a flat cone, maximum principle related to $R$ and $f$ is essentially used. 

\begin{proof}[Proof of part (a) of Theorem~\ref{thm:2}]
We use contradiction argument. 
If the statement of part (a) was wrong, there exists a sequence of  $(M_i, p_i, g_i,f_i) \in \mathcal{M}^*(A,H)$ such that
$f_i(p_i) \to 0$.  In light of equation (\ref{E101}), we have
\begin{equation}\label{E300}
   0 \leq \lim_{i \to \infty} R_i(p_i) \leq \lim_{i \to \infty} f_i(p_i)=0. 
\end{equation}
From Theorem \ref{T101}, by taking a subsequence if necessary,  we have
\begin{align}
    (M_i, p_i, g_i, f_i)  \longright{pointed-\hat{C}^{\infty}-Cheeger-Gromov} \left(M_{\infty}, p_{\infty}, g_{\infty}, f_{\infty} \right),  
\label{eqn:MH26_1}    
\end{align}
where $(M_{\infty}, p_{\infty}, g_{\infty}, f_{\infty})$ is an orbifold  Ricci shrinker. 
We claim that $R_{\infty}(p_{\infty})=0$.  Actually, if $p_{\infty}$ is a regular point, 
then it follows from smooth convergence around $p_{\infty}$ and (\ref{E300}) that $R_{\infty}(p_{\infty})=0$.
Therefore, we only need to study the case that $p_{\infty}$ is singular.  However, taking limit of equation (\ref{E101}), we know
$R_{\infty}+|\nabla f_{\infty}|^2=f_{\infty}$ on the regular part of $M_{\infty}$.  Lifting this equation to the orbifold chart around $p_{\infty}$, we obtain
\begin{align}
   \tilde{R}_{\infty} + |\widetilde{\nabla f_{\infty}}|^2=\tilde{f}_{\infty}, \quad \textrm{on} \; B(0, \delta) \backslash \{0\}.     \label{eqn:MH26_4}
\end{align}
Since (\ref{E101}) implies that $\left|\nabla \sqrt{f} \right| \leq \frac{1}{2}$ uniformly, it is clear that $f_{\infty}(p_{\infty})=0$ and $\tilde{f}_{\infty}(0)=0$, where $0$ is the image of 
$p_{\infty}$ in the local orbifold chart description.   On $B(0, \delta) \backslash \{0\}$, we also have $\Delta \tilde{f}_{\infty}=\frac{n}{2}-\tilde{R}_{\infty}$ in classical sense.
As $\tilde{R}_{\infty}$ is a smooth function on $B(0,\delta)$ and $\tilde{f}_{\infty}$ has bounded $W^{1,2}$-norm(c.f. (\ref{eqn:MH26_4})), we know
$\Delta \tilde{f}_{\infty}=\frac{n}{2}-\tilde{R}_{\infty}$ holds on $B(0, \delta)$ in the distribution sense. 
Consequently,  standard elliptic equation theory implies that $\tilde{f}_{\infty}$ is a smooth function across the singularity. 
Therefore, (\ref{eqn:MH26_4}) is improved to hold across singularity:
\begin{align}
   \tilde{R}_{\infty} + |\widetilde{\nabla f_{\infty}}|^2=\tilde{f}_{\infty}, \quad \textrm{on} \; B(0, \delta).     \label{eqn:MH26_5}
\end{align}
However, note that $\tilde{f}_{\infty}$ is $\Gamma$-invariant, where $\Gamma$ is the local orbifold group fixing the point $0$.    The smoothness of $\tilde{f}_{\infty}$ implies that
$|\widetilde{\nabla f_{\infty}}|(0)=0$. Plugging this equality and the fact that $\tilde{f}_{\infty}(0)=0$  into (\ref{eqn:MH26_5}), we obtain $\tilde{R}_{\infty}(0)=0$, which means that
$R_{\infty}(p_{\infty})=0$.  Therefore, no matter whether $p_{\infty}$ is a regular point, we have proved that $R_{\infty}(p_{\infty})=0$, as claimed. 
Since $f_{\infty}$ achieves a minimum $0$ at point $p_{\infty}$, the fact $R_{\infty}(p_{\infty})=0$ forces that $(M_{\infty}, g_{\infty})$ is a flat cone by Theorem~\ref{thm:MH26_2}. 
Then (\ref{eqn:MH26_1}) reads as 
\begin{align}
    (M_i, p_i, g_i, f_i)  \longright{pointed-\hat{C}^{\infty}-Cheeger-Gromov} \left(\R^{4}/\Gamma, 0, g_{E}, f_{E}\right),  
\label{eqn:MH26_3}    
\end{align}
which contradicts Theorem~\ref{thm:1}. 
The proof of part (a) of Theorem~\ref{thm:2} is complete.
\end{proof}

\begin{proof}[Proof of part (b) of Theorem~\ref{thm:2}]
We use contradiction argument again. 
For otherwise, there exists a sequence $(M_i, p_i, g_i,f_i) \in \mathcal{M}^*(A,H)$ such that $R_i(q_i)$ tends to $0$  and $q_i \in B(p_i,1)$.
Using Theorem \ref{T101} again, by taking a subsequence if necessary,  we may assume (\ref{eqn:MH26_1}) holds. 
Furthermore, we can assume $q_i \to q_{\infty} \in M_{\infty}$. 

We claim that $R(q_{\infty})=0$.  This holds trivially if $q_{\infty}$ is a regular point, following from the smooth convergence around $q_{\infty}$ and the fact $R(q_i) \to 0$. 
Therefore we focus on the case that $q_{\infty}$ is singular.  By Theorem~\ref{thm:1} and Theorem~\ref{thm:MH26_2}, we know $(M_{\infty}, g_{\infty})$ is not a flat cone 
and $R(q_{\infty})>2\xi$ for some positive $\xi$ depending on $q_{\infty}$.  Recall that $R$ is a Lipschitz continuous function on $M_{\infty}$.
So we can choose a small radius $r$ such that $R>\xi$ on $\partial B(q_{\infty}, r)$ which contains no singularity.    
Note that from part (a) and the fact that $f_{\infty}$ achieves minimum value at $p_{\infty}$, we obtain $f_{\infty} \geq C_a$ everywhere. 
Letting $c$ be a small constant in $(0, 0.5 C_a \xi)$, we obtain $R f_{\infty} \geq 2c$ on $\partial B(q_{\infty}, r)$, which means that
\begin{align*}
   R - 2cf^{-1} \geq 0,  \quad \textrm{on} \; \partial B(q_{\infty}, r), \quad \Rightarrow \quad R-cf^{-1} \geq 0, \quad \textrm{on} \; \partial B(q_i, r). 
\end{align*}
Note that the value of $R-cf^{-1}$ at point $q_i$ is strictly less than $0$ for large $i$.  Therefore, the minimum value of $R-cf^{-1}$ on $B(q_i,r)$ is achieved at some interior point $q_i'$, where we have
\begin{align*}
 0 \leq \Delta_f \left( R-cf^{-1}\right)= \left( R-cf^{-1}\right) -2|Rc|^2 + 2Rf^{-3} \leq \left( R-cf^{-1}\right) + 2Rf^{-3},
\end{align*}
which yields that
\begin{align*}
    R(q_i') \geq \frac{cf^2}{2+f^3}(q_i').  
\end{align*}
Since $q_i'$ is the minimum value of $R-cf^{-1}$, we have
\begin{align*}
   R(q_i)-cf^{-1}(q_i) \geq R(q_i')-cf^{-1}(q_i') \geq \frac{cf^2}{2+f^3}(q_i') -cf^{-1}(q_i')=\frac{-2c^2}{f(q_i') \left( f^3(q_i') +c\right)}. 
\end{align*}
The above inequality can be rewritten as
\begin{align*}
   c^{-1}R(q_i) \geq f^{-1}(q_i) -\frac{2c}{f(q_i') \left( f^3(q_i') +c\right)}. 
\end{align*}
By the fact $C_a \leq f \leq C(A)$ point-wisely in $B(p_i,1)$, we can choose $c$ very small such that the right hand side of the above inequality
is bounded below by some positive $C_0$ depending only on $A$.  This forces that
\begin{align*}
  R(q_i) \geq C_0 c
\end{align*}
for every large $i$, which contradicts our assumption that $R(q_i) \to 0$.  Therefore, no matter whether $q_{\infty}$ is a regular point, we have
proved that $R(q_{\infty})=0$. 

Now we can follow the route of the proof of part (a) to obtain a contradiction by the combination of Theorem~\ref{thm:MH26_2} and Theorem~\ref{thm:1}. 
Thus, the proof of part (b) of Theorem~\ref{thm:2} is complete.
\end{proof}

\begin{proof}[Proof of part (c) of Theorem~\ref{thm:2}]

We first remark that part (b) of Theorem~\ref{thm:2} can be stated in a more general manner. 
In particular, we have $R \geq C_{b,r}$ on each $B(p,r)$, where $C_{b,r}$ is a positive constant depending only on $A$ and $r$. 
Therefore, it follows from   (\ref{E106}) in Lemma~\ref{L100} that $R$ has uniform lower bound on the set $\{x | f(x) =9\}$. 
So we can choose a uniform constant $c=c(A,H)$ such that $u=R-cf^{-1}-4cf^{-2}>0$ on the set $\{x | f(x) =9\}$.
By inequalities (\ref{E106}) again and the fact that $R \geq 0$, we have $\displaystyle \lim_{x \to \infty} u(x) \geq 0$. 

We claim that $u \geq 0$ on the set $\{x | f(x)  \geq 9\}$. For otherwise, $u$ achieve a negative minimum at some point $x_0 \in \{x| f(x)>9\}$.  
At the point $x_0$, we can apply maximum principle(c.f. inequality (6) of Chow-Lu-Yang~\cite{BLY11}) to obtain
\begin{align*}
 0 \leq \Delta_fu \le u-4cf^{-3}\left(\frac{f}{4}-2 \right)-cf^{-4}(2f+24)|\nabla f|^2. 
\end{align*}
Recall that $f(x_0)>9$. Then the above inequality yields that
\begin{align*}
 cf^{-3}(x_0) <  \left. 4cf^{-3}\left(\frac{f}{4}-2 \right)+cf^{-4}(2f+24)|\nabla f|^2 \right|_{x_0} \leq u(x_0) <0,
\end{align*}
which is absurd since $f$ is positive at $x_0$.  Therefore, we have proved $u \geq 0$ on $\{x | f(x)  \geq 9\}$, as claimed.  

The fact that $u \geq 0$ on $\{x | f(x)  \geq 9\}$ means that
\begin{align}
  Rf \geq c + 4cf^{-1} \geq c_0, \quad \textrm{on} \; \{x | f(x) \geq 9\},    \label{eqn:MH27_1}
\end{align}
for some $c_0=c_0(A,H)$, where we have used the upper bound of $f$ by  (\ref{E106}). 
On the other hand, by (\ref{E106}) again, the set $\{x|f(x)<9\}$ is a uniformly bounded set.  It follows from part (a) and the generalized version of part (b) that 
\begin{align}
  Rf \geq c_0',  \quad \textrm{on} \; \{x | f(x) < 9\}    \label{eqn:MH27_2}
\end{align}
for some uniform positive constant $c_0'=c_0'(A,H)$.    Letting $C_{c}=\min\{c_0, c_0'\}$, the combination of (\ref{eqn:MH27_1}) and (\ref{eqn:MH27_2}) implies that
$Rf \geq C_c$ on the whole $M$.     
Therefore, the proof of part (c) of Theorem~\ref{thm:2} is complete.  
\end{proof}

\section{Further Questions}

Inspired by Theorem~\ref{thm:1}, we believe the following statement is true. 

\begin{conj}
 For each positive integer $n \geq 4$,   there exists a uniform constant $\epsilon=\epsilon(n)$ such that 
 for each non-flat Ricci shrinker $(M,p, g,f)$ and each finite subgroup $\Gamma \subset O(n)$ acting freely on $S^{n-1}$, the following estimate hold:
  \begin{align*}
  d_{PGH} \left\{ (M,p,g), (\R^n/\Gamma,0,g_{E})\right\}>\epsilon. 
\end{align*} 
\label{cje:MH23_1}
\end{conj}
Note that in Conjecture~\ref{cje:MH23_1}, $\epsilon$ does not depend on $\Gamma$ and depends only on $n$. Therefore, one has to deal with the collapsing case. 
More generally, one may also replace the cone $\R^n/\Gamma$ by $\R^{k}/\Gamma \times \R^{n-k}$ for some $k \geq 4$ and $\Gamma \subset O(k)$.
Replacing the Euclidean space by cylinder, we also make the following conjecture. 
\begin{conj}
 For each positive integers $n \geq 4$ and $n \geq k \geq 3$,   there exists a uniform constant $\epsilon=\epsilon(n)$ such that 
 for each non-flat, non-Eisntein Ricci shrinker $(M,p, g,f)$ and each finite subgroup $\Gamma \subset O(k)$ acting freely on $S^{k}$, the following estimate hold:
  \begin{align*}
  d_{PGH} \left\{ (M,p,g), ((S^k/\Gamma) \times \R^{n-k}, 0,g_{c})\right\}>\epsilon, 
\end{align*} 
where $g_c$ is the cylindrical metric, i.e., product metric of $g_{round} \times g_E$. 
\label{cje:MH23_11}
\end{conj}

Motivated by Theorem~\ref{thm:2},  we make the following conjecture.  

\begin{conj}
 For each positive integer $n \geq 4$,   there exists a uniform constant $\epsilon=\epsilon(n)$ such that 
 for each non-flat Ricci shrinker $(M,p, g,f)$ we have
 \begin{align*}
    \sup_{x \in M} R(x) \geq \epsilon. 
 \end{align*} 
\label{cje:MH25_1}
\end{conj}

Similar statements can be asked for $f$ and $Rf$, under the normalization condition (\ref{E101}). We leave these generalizations to interested readers. 
We close this article by returning to our initial question.

\begin{quest}
 Does there exist a sufficient and necessary condition for orbifold 4d Ricci shrinkers being able to be approximated by smooth Ricci shrinkers?
\label{qst:CA07_1} 
\end{quest}

\vskip10pt

Yu Li, Department of Mathematics, University of Wisconsin-Madison,
Madison, WI 53706, USA;  yli427@wisc.edu.\\

Bing  Wang, Department of Mathematics, University of Wisconsin-Madison,
Madison, WI 53706, USA;  bwang@math.wisc.edu.\\

\end{document}